\documentclass[12pt]{amsart}
\usepackage{amsmath,amssymb,amsthm,pifont}
\usepackage{enumitem}
\usepackage{graphicx}
\usepackage{amsfonts}
\usepackage{bigints}
\usepackage{amsfonts}
\usepackage{mathptmx}
\usepackage{hyperref}

\usepackage{epstopdf}

\newcommand{\NN}{{\mathbb N}} 

\theoremstyle{plain}
\newtheorem{theorem}{Theorem}[section]
\newtheorem{proposition}[theorem]{Proposition}

\newtheorem{lemma}[theorem]{Lemma}
\theoremstyle{definition}

\newtheorem{example}[theorem]{Example}
\newtheorem{remark}[theorem]{Remark}

\theoremstyle{Conjecture}

\makeatletter
\@addtoreset{equation}{section}
\makeatother

\makeatletter
\newcommand{\Spvek}[2][r]{%
	\gdef\@VORNE{1}
	\left(\hskip-\arraycolsep%
	\begin{array}{#1}\vekSp@lten{#2}\end{array}%
	\hskip-\arraycolsep\right)}

\def\vekSp@lten#1{\xvekSp@lten#1;vekL@stLine;}
\def\vekL@stLine{vekL@stLine}
\def\xvekSp@lten#1;{\def\temp{#1}%
	\ifx\temp\vekL@stLine
	\else
	\ifnum\@VORNE=1\gdef\@VORNE{0}
	\else\@arraycr\fi%
	#1%
	\expandafter\xvekSp@lten
	\fi}
\makeatother

\begin{document}
	
	\title{Paw{\l}ucki-Ple\'{s}niak extension operator for Non-Markov Sets}

	\author{Alexander Goncharov}
	\address{Department of Mathematics, Bilkent University, 06800 Ankara, Turkey}
	\email{goncha@fen.bilkent.edu.tr}
	\author{Yaman Paksoy}
	\address{Department of Mathematics, Bilkent University, 06800 Ankara, Turkey}
	
	\email{yamanpaks@hotmail.com}

	\date{Received: date / Accepted: date}

	\begin{abstract}
		We show that Paw{\l}ucki-Ple\'{s}niak's operator is bounded for some non-Markov sets.
     \end{abstract}
		\maketitle
	\section{Introduction}

By Whitney's extension theorem \cite{W}, for each compact set $K\subset {\Bbb R}^d,$ there is a continuous linear operator
extending jets of finite order from  ${\mathcal E}^p(K)$ to functions defined on the whole space, preserving the order of differentiability.
However, for $p=\infty,$ such an operator does not exist in the general case. Let us say that $K$  has the {\it extension property (EP)}
if there exists a linear continuous extension operator  $W:{\mathcal E}(K) \longrightarrow C^{\infty}({\Bbb R}^d)$. By Tidten (\cite{T}),
a set $K$ has $EP$ if and only if the space ${\mathcal E}(K)$ possesses the dominating norm. By \cite{GU}, there is no complete description
of $EP$ in terms of  densities of measures, Hausdorff contents or related characteristics.
For a short review of known extension operators  we refer the reader to Section 2 in \cite{GU}.

One of the approaches is due to W. Paw{\l}ucki and W. Ple\'{s}niak. In \cite{PP}, they present an extension operator $W$ in the form of a telescoping series
containing Lagrange interpolating polynomials with Fekete nodes. The operator initially was considered for uniformly polynomially cuspidal compact sets.
Later, in \cite{P}, the result was extended to any Markov set. By T.3.3 in \cite{P}, for each $C^{\infty}$ determining compact set $K$, the suggested
operator is continuous in the so-called Jackson topology  $\tau_J$ (see \cite{P} and Section 2 in \cite{AG} for the definition and some properties of
$\tau_J$) if and only if $\tau_J$ coincides with the natural topology  $\tau$ of the space ${\mathcal E}(K)$ and this happens if and only if the set $K$
is Markov. Since $\tau_J$ is not stronger then  $\tau,$  each Markov set has $EP$. However, by  \cite{ag96} and  \cite{ag97},
the inverse implication is not valid. Thus, in the case of non-Markov compact set with $EP$, the Paw{\l}ucki-Ple\'{s}niak extension operator is not continuous in $\tau_J$, yet this does not exclude the possibility for it to be bounded in $\tau$. Our aim is to check continuity of the operator in the
natural topology of the space.

In the construction of the extension operator in \cite{P}, Ple\'{s}niak used four main components:\\
1) Markov property of the set $K$, which implies\\
2) the possibility of a suitable individual extension of polynomials from the set $K$ to some neighborhood of the set, while preserving the norm of polynomials\\
3) a moderate growth of the Lebesgue constants corresponding to interpolation at Fekete points\\
4) the classical Jackson theorem on the polynomial approximation of $C^{\infty}-$ functions on the interval.\\

We analyze these components for non-Markov sets with $EP$. The paper is organized as follows. In Section 2, we repeat the argument from  \cite{P} for a
Markov set $K$ on the line such that $K$ contains an array of interpolating nodes with the polynomial growth of the corresponding Lebesgue constants.
Section 3 considers the family $K^\alpha, \alpha>1,$ of non-Markov Cantor-type sets. They have $EP$ provided  $\alpha\leq 2$.
Endpoints of intervals in Cantor procedure are enumerated in a special way. Some properties of the resulting sequence $(x_k)_{k=1}^{\infty}$ are considered, which makes it possible to estimate the corresponding Lebesgue constants.
In Section 4, we introduce the Markov $M_N^{(p)}$ factor as the norm of the operator of $p$-fold differentiation in the space of polynomials of degree at most $N$. In contrast to the Markov case, for small sets, the value of $M_N^{(p)}$ is substantially less than the $p-$th power of the usual Markov factor $M_N.$ The fast growth of $M_N^{(p)}$ can be neutralized (Section 5) by an ultra-fast rate of polynomial approximation of functions from
${\mathcal E}(K^\alpha)$. In Section 6 we show that the boundedness of the operator depends not on the rate of growth of $M_N^{(p)}$ factors
but rather on the possibility of suitable individual extensions of polynomials. In the considered case, the polynomials involved in the interpolation process form a topological basis in the space ${\mathcal E}(K)$. This creates a bridge to Mityagin's method of extension \cite{mi}.
Section 7 contains the main result: at least in the considered case,  Paw{\l}ucki-Ple\'{s}niak's operator is continuous in the natural topology
of the space. Moreover, Ple\'{s}niak's argument can be used for the given non-Markov set as well.

\section{Paw{\l}ucki-Ple\'{s}niak's operator for Markov sets}

For the convenience of the reader, we repeat the main points of \cite{P}. We shall restrict the discussion to a perfect compact subset $K$ of the line
with a closed interval $I$ containing $K$.

Let ${\mathcal P}_n$ denote the set of all polynomials of degree at most $n$ and $|\cdot|_K$ (later $|\cdot|_{0,K}$) stand for the uniform norm on $K.$
We say that $K$ is {\it Markov} (or $K$ has {\it Markov's property (MP)}) if there are constants $C_1$ and $r$ with
\begin{equation} \label{MP}
|Q^{(p)}|_K \leq  C_1 n^{rp} \,|Q|_K \,\,\,\,\,\mbox {for all}\,\,\,\,\,\,Q\in {\mathcal P}_n, p\in \Bbb N.
\end{equation}
Given $\delta>0$, let $K_{\delta}$ denote $\{x\in \Bbb R: \mathrm{dist}(x,K)\leq \delta\}$ and $u_{\delta}$ be a $C^{\infty}$ function with the properties:
$0\leq u_{\delta}\leq 1, u_{\delta}=1$ on $K$, $u_{\delta}=0$ on $\Bbb R\setminus K_{\delta},$ and $|u^{(j)}_{\delta}(x)|\leq c_j \delta^{-j}$ for each
$j\in \Bbb N, x\in \Bbb R.$ The constants $c_j$ do not depend on $K$. Although the existence of such a function is well known, a specific version of $u_{\delta}$ is considered in Section 6. Suppose $K$ satisfies \eqref{MP}. For each polynomial $Q$ of degree $n$ let us take $\delta=n^{-r}.$
By means of Taylor's expansion, it's easy to show that
\begin{equation} \label{IE}
|Q(x)| \leq C_1\,e\, |Q|_K\,\,\,\,\mbox {for}\,\,\,\,\,Q\in {\mathcal P}_n, x\in K_{n^{-r}}.
\end{equation}
In addition, if $p\leq n$ then for each $x\in \Bbb R$ we have
\begin{equation} \label{IE1}
|(Q\cdot u_{\delta})^{(p)}(x)|\leq C_2 n^{rp}\, |Q|_K,
\end{equation}
where $ C_2$ depends only on $p$. Indeed, it is evident for $x\in K$ or $x\notin K_{\delta}$. Otherwise, by Leibnitz's rule,
$$|(Q\cdot u_{\delta})^{(p)}(x)|\leq \sum_{j=0}^p  \binom{p}{j}c_j n^{rj} C_1e|Q^{(p-j)}|_K\leq C_2 n^{rp}|Q|_K$$ with $C_2:=C_1^2e\sum_{j=0}^p\binom{p}{j}c_j.$

Let $\mathcal X=(x_{k,n})_{k=1, n=1}^{n, \infty}$ be an infinite triangular matrix of points from $K$ such that each row $\mathcal X_n$ consists of distinct elements. For a fixed $n$, the points of $\mathcal X_n$  determine the polynomial $\omega_{n}(x) = \prod_{k=1}^{n}(x-x_{k,n})$,
the fundamental Lagrange polynomials $l_{k,n}(x)= \frac{\omega_{n}(x)}{(x-x_{k,n})\omega'_{n}(x_{k,n})}$ with $1\leq k\leq n$, the Lebesgue function
$\lambda_n(x)=\sum _{k=1}^{n}|l_{k,n}(x)|$, and the Lebesgue constant $\Lambda_n(K)=\sup_{x\in K} \lambda_n(x).$ Given function $f$ defined on $K$, by $L_n(f,\cdot;\mathcal X_n)$ we denote the corresponding Lagrange interpolating polynomial, by $L_n(\cdot, \mathcal X_n)$ the interpolating projection,
so $\Lambda_n(K)$ is the sup-norm  of $L_n(\cdot, \mathcal X_n)$ in $C(K)$. Here, $ l_{k,n}, L_n\in {\mathcal P}_{n-1},$ so the index $n$ of the Lebesgue constant corresponds to the number of interpolating points. Suppose $\mathcal X$ is chosen in a such way that the sequence $(\Lambda_n(K))_{n=1}^{\infty}$ has at most polynomial growth: there are constants $C_3$ and $R$ such that
\begin{equation} \label{LC}
\Lambda_{n+1}(K)\leq C_3 n^R  \,\,\,\,\,\mbox {for all}\,\,\,\,\,n\in \Bbb N.
\end{equation}
Of course, Fekete points provide this condition with $R=1$, see for instance Section 2 in \cite{P}.

Our main object is the operator $W:{\mathcal E}(K) \longrightarrow C^{\infty}({\Bbb R})$ which is defined as follows
\begin{equation} \label{W}
W(f,x)= L_1(f,x;\mathcal X_1)\cdot u_{\delta_1}(x)+\sum _{n=1}^{\infty}[L_{n+1}(f,x;\mathcal X_{n+1})- L_n(f,x;\mathcal X_n)]\cdot u_{\delta_n}(x).
 \end{equation}

Recall that, by \cite{W}, the Whitney space ${\mathcal E}(K)$ consists of traces on $K$ of functions from $C^{\infty}(I)$ and the Whitney topology $\tau$ is given by the seminorms
$$ \|\,f\,\|_{q} = |f|_{q,K} + sup \left\{ |(R_{y}^{q} f)^{(k)}(x)|\cdot |x-y|^{k-q} : x,y \in K ,x \neq y, k = 0,1,...q \right\} $$
for $q\in {\Bbb Z}_+, $   where $|f|_{q,K} = sup \{ |f^{(k)} (x)| : x \in K, k \leq q \} \text { and } R_{y}^{q}f(x) = f(x) - T_{y}^{q} f(x)$ is the Taylor remainder. By the open mapping theorem, for any $q$ there exists $C$ such that
\begin{equation}\label{quot}
 inf \,\, |\,F\,|_{\,q,I}\leq C\,||\,f\,||_{\,q},
\end{equation}
where the infimum is taken over all extensions of $f$ to $F\in C^{\infty}(I),$ see, e.g., (2.3) in \cite{basW}.\\

Following Zerner \cite{Z}, let us consider the given below seminorms in  ${\mathcal E }(K)$
$$ d_{-1}(f) = |f|_{K},\,\,d_{0}(f) = E_0(f,K), \,\, d_{k}(f) = \sup_{n\geq 1} n^k\,E_n(f,K)\,\,\mbox{for}\,\, k\geq 1,$$
where $E_n(f,K):=\min_{P\in {\mathcal P}_n}|f-P|_K$ is the best approximation of $f$ by $n-$th degree polynomials. Since $K$ is perfect,
the  {\it Jackson topology} $\tau_J$, given by $(d_k)$, is Hausdorff. By Jackson's theorem, $\tau_J$ is well-defined and is not stronger than
$\tau$ (see \cite{AG} for more details).  By Ple\'{s}niak (T.3.3 in \cite{P}), the operator $W$ is continuous in
 $\tau_J \Leftrightarrow  \tau_J=\tau \Leftrightarrow ({\mathcal E}(K), \tau_J)$ is complete $\Leftrightarrow K$ is Markov.
 We now extract a part of this theorem in the following form.
\begin{theorem}(\cite{P}) \label{WP}
Let  $K\subset \Bbb R$ be Markov and the operator $W$ be given by an array $\mathcal X$ satisfying \eqref{LC}. Then $W$ is bounded in $\tau$.
\end{theorem}
\begin{proof} Let $G_n$ denote the $n-$th term of the series \eqref{W} with $\delta_n=n^{-r},$ where $r$ is taken from \eqref{MP}.
Fix $p\in \Bbb N$ and $x\in  \Bbb R$.  The expression in square brackets is a polynomial of degree $n$, so, by \eqref{IE1},
$|G_n^{(p)}(x)|\leq C_2 n^{rp}\, |L_{n+1}(f)- L_n(f)|_K.$ Lebesgue's Lemma (see, e.g., \cite{VL}, Chapter 2, Prop. 4.1) now yields
$|L_{n+1}(f)- L_n(f)|_K\leq |L_{n+1}(f)-f|_K +|L_n(f)- f|_K\leq (\Lambda_{n+1}(K)+1)E_n(f,K)+ (\Lambda_n(K)+1)E_{n-1}(f,K).$
By \eqref{LC}, $|G_n^{(p)}(x)|\leq C_4 n^{rp+R}E_{n-1}(f,K),$  where $C_4$ does not depend on $n$ and $f$. Clearly,
$E_{n-1}(f,K)\leq E_{n-1}(F,I),$ where $F$ is any extensions of $f$ to $F\in C^{\infty}(I).$ Jackson's Theorem (see, e.g., \cite{VL}, Chapter 7, Cor. 6.5)
and \eqref{quot} show that $|G_n^{(p)}(x)|$ is a term of a uniformly convergent series, which completes the proof.
 \end{proof}

We aim to show that a modified version of this theorem can also be applied to some non-Markov sets.

\section{Non-Markov sets with the extension property}

 We consider the family $K^\alpha$ of sets proposed in \cite{ag97}. A geometrically symmetric Cantor set $K$ is the intersection
 $ \bigcap\limits_{s=0}^\infty E_s$, where $E_0 = [0,1]$ and for each $s\in \NN$, the set $E_s$ is a union of $2^s$ closed {\it basic} intervals $I_{j,s}$, $j=1,2,...,2^s$ of length $\ell_s$. Recursively, $E_{s+1}$ is obtained by replacing each interval $I_{j,s}$ by two {\it adjacent} subintervals $I_{2j-1,s+1}$ and $I_{2j,s+1},$ where the distance between them is $h_s= \ell_s - 2\ell_{s+1}$.

For $\alpha > 1$ and $\ell_1 < 1/2$ with $2\ell_1^{\alpha-1}<1$,  the set $K^\alpha$ is the Cantor set associated to the lengths of intervals satisfying
$\ell_s = \ell_{s-1}^\alpha = \ell_1^{\alpha^{s-1}}$ for $s\geq 2$. By \cite{Pl}, $K^\alpha$ are not Markov. On the other hand, by \cite{ag97} and \cite{A=2},
the set $K^\alpha$ has the extension property if and only if $\alpha\leq 2.$

We follow the notation of \cite{QE}. Let $X_0:=\{0, 1\}$ and, for $k\in \Bbb N,$ let $X_k$ be the set of endpoints of intervals from $E_k$ that are not endpoints of intervals from $E_{k-1}.$ Thus, $X_1:=\{\ell_1, 1-\ell_1\}, X_2:=\{\ell_2, \ell_1-\ell_2, 1-\ell_1+\ell_2, 1-\ell_2\},$ etc.
We refer $s-${\it th type} points to the elements of $X_s.$
Set $Y_s=\cup_{k=0}^s X_k.$ Clearly, $\#(X_s)=2^s$ for $s\in \Bbb N$ and $\#(Y_s)=2^{s+1}$ for $s\in \Bbb Z_+.$ Here and below, $\#(Z)$ denotes the
cardinality of a finite set $Z$.

Let $Z=(z_k)_{k=1}^{M}\subset \Bbb R$ and $\omega_M(x) = \prod_1^M(x-z_k).$ For a fixed $x\in \Bbb R,$ by $d_k(x,Z)$ we denote the distances $|x-z_{j_k}|$ from $x$ to points of $Z$, where these distances are arranged in the nondecreasing order, so $d_k(x,Z)\leq d_{k+1}(x,Z)$ for $k = 1,2,\cdots, M-1$.
Then $|\omega_M(x)|=\prod_{k=1}^M  d_k(x,Z)$ and, given $p<M$, the $p-$th derivative of $\omega_M$ at the point $x$ is the sum of $\frac{M!}{(M-p)!}$ products, where each product contains $M-p$ terms of the type $(x-z_k)$. Hence
\begin{equation}\label{om}
|\omega_M^{(p)}(x)|\leq M^p\,\prod_{k=p+1}^M d_k(x,Z).
\end{equation}

Suppose we are given a finite set $Z=(z_k)_{k=1}^{M}\subset K^\alpha$. Let $m_{j,s}(Z):=\#(Z\cap I_{j,s}).$ We say that points of $Z$ are {\it uniformly distributed} on $K^\alpha$ if for each $k\in \Bbb N$ and $i,j \in \{1,2, \ldots, 2^k\}$ we have
\begin{equation}\label{un}
|m_{i,k}(Z)- m_{j,k}(Z)|\leq 1.
\end{equation}
As in \cite{GU}, see also \cite{QE}, we put all points from $\cup_{k=0}^{\infty} X_k$  in order by {\it the rule of increase of type}.
First, we enumerate points from
$Y_0=X_0: x_1=0, x_2 =1.$ After this we include points from $X_1$ by increase the index of each point by 2: $x_3= \ell_1, x_4 = 1-\ell_1.$
Thus, $Y_1$  in ascending order is $\{x_1,x_3,x_4,x_2\}.$ Increasing the index of each point by 4 gives the order $X_2=\{x_5,x_7,x_8,x_6\}$
with $Y_2=\{x_1,x_5, x_7, x_3,x_4, x_8, x_6, x_2\}.$
Continuing in this fashion, we use $Y_{k-1}=\{x_{i_1},x_{i_2},\cdots, x_{i_{2^k}}\}$ to define the ordering $X_k=\{x_{i_1+2^k},x_{i_2+2^k},\cdots, x_{i_{2^k}+2^k}\}.$ We see that the points with odd indices are on the left part of $K^\alpha$, whereas $x_{2n}\in I_{2,1}.$ For each $M$,
the first $M$ points chosen by the above rule are uniformly distributed on $K^\alpha$.\\

Let us consider the location of these points in more detail. Suppose $k\in \Bbb N$ and $1\leq j \leq 2^k.$ Then  $(x_i)_{i=1}^{2^k}$ is $Y_{k-1}$ that
is the set of all endpoints of types $\leq k-1$, whereas $x_{2^k+j}=x_j\pm\ell_k,$ where the sign is determined as follows: if $2^k+j=2^k+2^m+\cdots +2^r+1$
with $\kappa:=\#\{k, m, \ldots, r\}$ then the sign is $(-1)^{\kappa-1}.$

In what follows, we will interpolate polynomials at points of $Z=(x_k)_{k=1}^{N+1}$: if $Q\in {\mathcal P}_{N}$ then
\begin{equation}\label{int}
Q(x)=\sum_{k=1}^{N+1} Q(x_k) \frac{a_k(x)}{a_k(x_k)},
\end{equation}
where $a_k(x)=\prod_{i=1, i\neq k}^{N+1}(x-x_i).$ Let us fix $N$ with the binary decomposition for $N+1$
\begin{equation}\label{NN}
N+1=2^s+ 2^{s_1}+2^{s_2}+\cdots +2^{s_m} \,\,\,\,\mathrm{with} \,\,\,\,0\leq s_m<\cdots < s_1 < s_0:=s.
\end{equation}
By the above,
\begin{equation}\label{N+2}
x_{N+2}=\ell_{s_m}-\ell_{s_{m-1}}+\ell_{s_{m-2}}-\cdots+(-1)^m \ell_{s_0}.
\end{equation}

The representation \eqref{NN} gives the decomposition
\begin{equation}\label{ZZ}
Z=A_s\cup A_{s_1} \cup \cdots \cup A_{s_m}
\end{equation}
with $\#(A_{s_j})=2^{s_j}$ for $0\leq j\leq m.$ Here, $A_s=(x_k)_{k=1}^{2^s}=(x_{k,s})_{k=1}^{2^s},$ with $x_{1,s}=0<x_{2,s}=\ell_{s-1}<\cdots <x_{2^s,s}=1.$ Similarly,
$ A_{s_1}=(x_k)_{k=2^s+1}^{2^s+2^{s_1}}=(x_{k,s_1})_{k=1}^{2^{s_1}} ,\cdots, A_{s_m}=(x_k)_{k=N-2^{s_m}+2}^{N+1}=(x_{k,s_m})_{k=1}^{2^{s_m}}$ with $x_{k,s_j}\nearrow$ as $k$ increases and $j$ is fixed.

Our next objective is to determine the location of points from $A_{s_j}$.
Let $I_{i,k}=[a_{i,k},b_{i,k}]$ for $1\leq i\leq 2^k.$ Then $A_s=Y_{s-1}=\bigcup_{i=1}^{2^{s-1}}\{a_{i,s-1},b_{i,s-1}\}.$

Each $I_{i,s_1-1}$ contains two points of $A_{s_1}$ that are symmetric with respect to the endpoints of $I_{i,s_1-1}$. For example, $x_{1,s_1}=x_{2^s+1}=\ell_s$ is symmetric to $x_{2,s_1}=\ell_{s_1-1}-\ell_s,$ which is $x_{2^s+2^{s_1-1}+1}.$ The largest point of $A_{s_1}$, namely $x_{2^{s_1},s_1}=1-\ell_s$ is $x_k$ with $k=2^s+2=2^s+2^0+1.$ Thus, $A_{s_1}=\bigcup_{i=1}^{2^{s_1}-1}\{a_{i,s_1-1}+\ell_s,b_{i,s_1-1}-\ell_s\}.$
Likewise, $x_{1,s_2}=x_{2^s+2^{s_1}+1}=\ell_{s_1}-\ell_s$
is symmetric in regard to $I_{1,s_2-1}$ with  $x_{2,s_2}=\ell_{s_2-1}-\ell_{s_1}+\ell_s=x_{2^s+2^{s_1}+2^{s_2-1}+1},$ so
$A_{s_2}=\bigcup_{i=1}^{2^{s_2-1}}\{a_{i,s_2-1}+\ell_{s_1}-\ell_s,b_{i,s_2-1}-\ell_{s_1}+\ell_s\}.$ Continuing in this fashion, we get
$A_{s_m}=\bigcup_{i=1}^{2^{s_m}-1}\{a_{i,s_m-1}+\ell_{s_{m-1}}-\ell_{s_{m-2}}+\cdots -(-1)^{m} \ell_s,b_{i,s_m-1}- \ell_{s_{m-1}}+\ell_{s_{m-2}}+\cdots +(-1)^{m} \ell_s\}.$

Each $x\in K^{\alpha}$ determines the increasing chain of basic intervals:
\begin{equation}\label{x}
x\in I_{j,s} \subset  I_{j_1,s-1} \subset I_{j_2,s-2} \subset \cdots \subset I_{j_s,0}=[0,1].
\end{equation}
Let $J_{s}$ and $I_{j,s}$ be the  adjacent subintervals of $I_{j_1,s-1}$ and, more generally, for $1\leq n \leq s-1$, let
$J_n:=(I_{j_{s-n+1},n-1}\setminus I_{j_{s-n},n})\cap E_n$. Obviously, $\#(A_s\cap I_{j,s})=\#(A_s\cap J_s)=1,\ldots, \#(A_s\cap J_n)=2^{s-n},$
so $\prod_{x_k\in A_s}|x-x_k|\leq \pi_0:=\ell_{s}\,\ell_{s-1}\,\ell_{s-2}^2 \cdots \ell_0^{2^{s-1}}.$

The same reasoning applies to  $1\leq j\leq m.$ Let $\pi_j:=\ell_{s_j}\,\ell_{s_j-1}\,\ell_{s_j-2}^2 \cdots \ell_0^{2^{s_j-1}}.$ Then
$|\omega_{N+1}(x)|=\prod_{0\leq j\leq m} \prod_{x_k\in A_{s_j}}|x-x_k| \leq \prod_{j=0}^m\pi_j.$

We shall use two more representations of the product above
\begin{equation}\label{pi}
|\omega_{N+1}(x)|\leq \prod_{j=0}^m\pi_j=\prod_{i=0}^s  \ell_i^{\lambda_i}=\prod_{k=1}^{N+1}\rho_k,
\end{equation}
where $(\rho_k)_{k=1}^{N+1}$ are all terms $\ell_i$ of the product arranged in nondecreasing order. We see that the real multiplication in
$\prod_{i=0}^s  \ell_i^{\lambda_i}$ starts at $i=1$ because $\ell_0=1,$ so at least half of the terms $\rho_k$ are 1.

The inequality \eqref{pi} is exact in powers in the following sense
$$|\omega_{N+1}(x_{N+2})|\geq \prod_{i=0}^s  h_i^{\lambda_i}.$$
Indeed,  $\prod_{x_k\in A_s}|x_{N+2}-x_k|\geq h_{s}\,h_{s-1}\,h_{s-2}^2 \cdots h_0^{2^{s-1}}.$ Taking into account \eqref{N+2} and the location of points in
$A_{s_j},$ we have a similar estimate of $\prod_{x_k\in A_{s_j}}|x_{N+2}-x_k|$ for  $1\leq j\leq m.$
By Lemma 2.1 in \cite{QE}, the ratio $h_n/\ell_n$ increases with $n$. Hence, $h_n \geq \ell_n \cdot h_0.$ Therefore,
\begin{equation}\label{pi1}
h_0^{N+1} \prod_{i=0}^s  \ell_i^{\lambda_i}\leq |\omega_{N+1}(x_{N+2})|\leq \prod_{i=0}^s  \ell_i^{\lambda_i}.
\end{equation}

As above, let $Z=(x_k)_{k=1}^{N+1}$. Fix $k$ with $1\leq k \leq N+1$ and $I_{j,s}$ containing $x_k.$ Our next goal is to analyze the function $a_k.$
The chain \eqref{x} for $x_k$ determines intervals $(J_n)_{n=1}^s$ and degrees $\mu_n:=\#(J_{n+1} \cap Z\setminus\{x_k\})$ for
$0\leq n\leq s-1, \mu_s:=\#(I_{j,s} \cap Z\setminus\{x_k\}).$ Clearly, $\mu_s=m_{j,s}-1$ and  $\sum_{n=0}^s \mu_n=N.$
 Let us consider the interval $I_{j,s}$. If $m_{j,s}=1,$ then $\mu_s=0.$
Otherwise, $m_{j,s}=2.$ Here, by the choice procedure, $d_1(x_k,Z\setminus\{x_k\})=\ell_{s}.$
In both cases this gives the factor $\ell_{s}^{\mu_s}$ in $|a_k(x_k)|$ corresponding to $I_{j,s}$.
If $x_i\in J_{n+1}$ then $h_n\leq |x_k-x_i|\leq \ell_n.$ It follows that
\begin{equation}\label{pi2}
h_0^{N} \prod_{i=0}^s  \ell_i^{\mu_i}\leq |a_k(x_k)|\leq \prod_{i=0}^s  \ell_i^{\mu_i}.
\end{equation}

Suppose $x\in I_{j,s}$, so $x_k$ and $x$ have the same degrees $(\mu_i)_{i=0}^s$. As above, $|a_k(x)|\leq \prod_{i=0}^s\ell_i^{\mu_i}.$
Thus, if $x$ and $x_k$ are on the same interval of the $s-$th level, then $|a_k(x)|\leq  h_0^{-N} |a_k(x_k)|.$

 Let us show that a similar upper bound holds for each $x\in K^\alpha$.

\begin{lemma}  \label{max}
Given $N\in {\Bbb N},$ let $Z=(x_k)_{k=1}^{N+1}$ be chosen in $K^\alpha$ as above, $1\leq k\leq N+1$. Then
$$\max_{x\in K^\alpha} |a_k(x)|\leq h_0^{-N} |a_k(x_k)|.$$
\end{lemma}
\begin{proof}
Fix $k$ and $\tilde x$ such that $\max_{x\in K^\alpha} |a_k(x)|=|a_k(\tilde x)|.$ For  $x_k$ we take the chain \eqref{x} and degrees $(\mu_i)_{i=0}^s$
as above. On the other hand, the point $\tilde x$ determines the chain
$\tilde x\in \tilde I_{j,s} \subset  \tilde I_{j_1,s-1} \subset \cdots \subset \tilde I_{j_s,0}=[0,1]$ with the corresponding
$(\tilde J_n)_{n=1}^s$ and $\nu_n:=\#(\tilde J_{n+1} \cap Z\setminus\{x_k\})$ for $0\leq n\leq s-1, \nu_s(x):=\#(\tilde  I_{j,s} \cap Z\setminus\{x_k\}).$
We aim to show that for $1\leq i\leq s$
\begin{equation}\label{mu-nu}
\mu_s+\mu_{s-1}+\cdots +\mu_i\leq \nu_s+\nu_{s-1}+\cdots +\nu_i.
\end{equation}
Let $n$ be the largest level for each $I_{j_{s-n},n}=\tilde I_{j_{s-n},n}.$ Then the intervals $J_i$ and $\tilde J_i$ coincide for $0\leq i\leq n$
and $\mu_i=\nu_i$ for $0\leq i\leq n-1.$ The interval $I_{j_{s-n},n}$ contains $N-\sum_{q=0}^{n-1}\mu_q$ which is $\sum_{q=n}^{s}\mu_q.$ The same is valid
for $\nu_q$. Hence, \eqref{mu-nu} is valid for $i=n$ and we need only prove it for $i\geq n+1.$
We note that $x_k$ and $\tilde x$ are on different subintervals of the $n+1-$st level of $I_{j_{s-n},n}$, namely, $x_k\in \tilde J_{n+1}, \tilde x\in J_{n+1}.$

Let us show \eqref{mu-nu} for $i=s.$ Since the set $Z$ is distributed uniformly, each interval of $s-$th level may contain not less than one and not more that two points of $Z$. Thus, $1\leq \nu_s$ and $\mu_s\leq 1,$ as the point $x_k$ is excluded. Similarly, by \eqref{un},
$$\mu_s+\mu_{s-1}=m_{j_1,s-1}-1\leq \#(\tilde I_{j_1,s-1}\cap Z)=\#(\tilde I_{j_1,s-1}\cap Z\setminus\{x_k\})=\nu_s+\nu_{s-1},$$
 since $x_k\notin \tilde I_{j_1,s-1}.$ We can repeat the argument for other $i\geq n+1,$ since for such values
$i$ the interval $ \tilde I_{j_{s-i},i}$ does not contain $x_k$. Thus, \eqref{mu-nu} is valid for all $i$.

We proceed to show that the desired assertion follows from \eqref{mu-nu}. On the one hand,
$\max_{x\in K^\alpha} |a_k(x)|\leq \prod_{j=0}^s  \ell_j^{\nu_j}=\ell_1^{\nu_1+\alpha \nu_2+\cdots+ \alpha^{s-1} \nu_s}.$
On the other hand, $ |a_k(x_k)|\geq \prod_{j=0}^s  h_j^{\mu_j}=\prod_{j=0}^s  \ell_j^{\mu_j} \cdot \prod_{j=0}^s  (h_j/\ell_j)^{\mu_j},$ where, as above,
$h_j/\ell_j\geq h_0.$ Hence, $ |a_k(x_k)|\geq h_0^N \prod_{j=0}^s  \ell_j^{\mu_j}$ and it remains to prove that
\begin{equation}\label{l1}
\prod_{j=0}^s \ell_j^{\nu_j} \leq \prod_{j=0}^s  \ell_j^{\mu_j},
\end{equation}
or, what is equivalent, that
$$\mu_1+\alpha \mu_2+\cdots+ \alpha^{s-1} \mu_s\leq \nu_1+\alpha \nu_2+\cdots+ \alpha^{s-1} \nu_s.$$
The left side can be written as $(\mu_s+\mu_{s-1}+\cdots +\mu_1)+(\alpha-1)(\mu_s+\mu_{s-1}+\cdots +\mu_2)+\cdots +(\alpha^{s-1}-\alpha^{s-2})\mu_s.$
Similar representation for the right side and \eqref{mu-nu} completes the proof.
\end{proof}

One may conjecture that the coefficient $h_0^{-N}$ in the above lemma can be reduced. However, it cannot be replaced by a factor that increases
polynomially with $N$. Let's show this.

\begin{example}  \label{ex}
Let $N=2^s+2, \alpha=2, \ell_1\leq1/4$ and a constant $r$ be fixed. Then for $Z=(x_k)_{k=1}^{N+1}, k=N+1, y=\ell_2-\ell_s$ we have
$N^r\,|a_k(x_k)|<|a_k(y)|$ for large $s$.

Here, $x_k=\ell_1-\ell_s, a_k(x)= \prod_{j=1}^{N}(x-x_j).$ We proceed to show that $\frac{|a_k(x_k)|}{|a_k(y)|}$ is exponentially small (with respect to $N$)
for large enough $s$. If the set $(x_j)_{j=1}^{2^s+2}$ is decomposed in the form $Y_{s-1}\cup A$ with  $A=(x_j)_{j=2^s+1}^{2^s+2}=\{\ell_s, 1-\ell_s\},$ then
$|a_k(x)|=\prod_{x_j\in Y_{s-1}}|x-x_j|\prod_{x_j\in A}|x-x_j|.$ For the second part, we have
$\prod_{x_j\in A}\frac{|x_k-x_j|}{|y-x_j|}=\frac{(\ell_1-2\ell_s)(1-\ell_1)}{(\ell_2-2\ell_s)(1-\ell_2)}<M:=\frac{2}{\ell_1(1-\ell_2)}$ that does not depend on $s$. It remains to estimate $\prod_{x_j\in Y_{s-1}}\frac{|x_k-x_j|}{|y-x_j|}=\prod_{j=1}^{2^s}\frac{d_j(x_k)}{d_j(y)},$ where for brevity, we drop the argument
$Y_{s-1}$ in $d_j(x,Y_{s-1}).$ By symmetry, $d_j(x_k)=d_j(y)$ for $1\leq j\leq 2^{s-2},$ that is for the points $x_i$ on the nearest (to the argument of $d_j$)
interval of the second level. Let us take $\tilde y:=\ell_1-\ell_2+\ell_s,$ which is symmetric to $y$ with respect to $I_{1,1}.$ Then $d_j( \tilde y)=d_j(y)$ for
$2^{s-2}+1\leq j\leq 2^{s-1}.$ Since $x_k-\tilde y=\ell_2-2\ell_s,$ we get
$\prod_{j=2^{s-2}+1}^{2^{s-1}}\frac{d_j(x_k)}{d_j(y)}=\prod_{j=2^{s-2}+1}^{2^{s-1}}(1+\frac{\ell_2-2\ell_s}{d_j(y)})=
\prod_{i=1}^{2^{s-2}}(1+\frac{\ell_2-2\ell_s}{D_i}).$ Here and for the remaining two intervals of the 2-nd level we express the corresponding products
in terms of $D_i:=d_{2^{s-2}+i}$ with $1\leq i\leq 2^{s-2}.$ For these intervals $d_j(x_k)=d_j(y)-\ell_1+\ell_2.$

If $x_i\in I_{3,2}$ then $d_{2^{s-1}+i}=D_i+1-2\ell_1+\ell_s,$ so
$\prod_{j=2^{s-1}+1}^{3\cdot 2^{s-2}}\frac{d_j(x_k)}{d_j(y)}=\prod_{i=1}^{2^{s-2}}(1-\frac{\ell_1-\ell_2}{D_i+1-2\ell_1+\ell_2}).$

If $x_i\in I_{4,2}$ then $d_{3\cdot 2^{s-2}+i}=D_i+1-\ell_1$ and
$\prod_{j=3\cdot 2^{s-2}+1}^{2^s}\frac{d_j(x_k)}{d_j(y)}=\prod_{i=1}^{2^{s-2}}(1-\frac{\ell_1-\ell_2}{D_i+1-\ell_1}).$

 Therefore, $\prod_{j=1}^{2^s}\frac{d_j(x_k)}{d_j(y)}=\prod_{i=1}^{2^{s-2}}(1+\frac{\ell_2-2\ell_s}{D_i})(1-\frac{\ell_1-\ell_2}{D_i+1-2\ell_1+\ell_2})
(1-\frac{\ell_1-\ell_2}{D_i+1-\ell_1})$ with admissible values $\ell_1-2\ell_2+\ell_s\leq D_i\leq \ell_1-\ell_2+\ell_s.$
The general term of the product (we denote it briefly by $b_i$) consists of three parts. It can be increased only if $D_i$ is replaced by $\ell_1-2\ell_2$
in the first part and by the maximum $D_i$ in the 2nd and 3rd parts. Hence,
$b_i\leq \frac{\ell_1-\ell_2}{\ell_1-2\ell_2}\frac{1-2\ell_1+\ell_2+\ell_s}{1-\ell_1+\ell_s}\frac{1-\ell_2+\ell_s}{1+\ell_1-2\ell_2+\ell_s}.$
For large $s$, the right side of the expression is as close to
$\sigma:=\frac{1-\ell_1}{1-2\ell_1}\frac{1-2\ell_1+\ell_2}{1-\ell_1}\frac{1-\ell_2}{1+\ell_1-2\ell_2}$ as we want it to be.
A straightforward computation shows that $\sigma<1.$ Let  $\sigma<\sigma_1<1.$ Then $\frac{|a_k(x_k)|}{|a_k(y)|}< M \,\sigma_1^{2^{s-2}},$
which is the desired conclusion.
\end{example}

\begin{remark}
Here, $\ell_1$ may be arbitrary small.
\end{remark}
\begin{remark}
The example shows that the sequence $(x_n)_{n=1}^{\infty}$ is not Leja. Recall that a sequence $(y_n)_{n=1}^{\infty}\subset K$ has Leja's property if $|y_1-y_2|=\mathrm{diam}(K)$ and, once $y_1, y_2, \dots, y_n$ have been determined, $y_{n+1}$ is chosen so that it provides the maximum modulus of the polynomial $(x-y_1)\cdots(x-y_n)$ on $K$. In our case,  $|\omega_{2^s+2}(x_{2^s+3})|<|\omega_{2^s+2}(y)|.$ However, using the example technique, it can be shown
that $(x_n)_{n=1}^{\infty}$ is a Leja sequence for $K^\alpha$ if $\alpha>2.$
\end{remark}
\begin{remark}
Similarly, the set $(x_n)_{n=1}^M$ is not a Fekete $M-$tuple for $M=2^s+3$. Indeed, let $V(t_1,\cdots, t_M)$ be the Vandermonde determinant. In our case,
$\frac{|V(x_1,\cdots, x_M)|}{|V(x_1,\cdots, x_{M-1},y)|}=\frac{|\omega_{2^s+2}(x_M)|}{|\omega_{2^s+2}(y)|}<1,$ whereas a Fekete $M-$tuple must
realize the maximum modulus of the Vandermonde determinant.
\end{remark}

\section{Markov $M_N^{(p)}$ factors}

As a first application of Lemma \ref{max}, we can estimate the Lebesgue constants for Newton's interpolation at points $(x_n)_{n=1}^{\infty}$.

\begin{proposition}  \label{Leb}
Given $N\in {\Bbb N},$ let $(x_k)_{k=1}^{N+1}$ be chosen by the rule of increase of type and $\Lambda_{N+1}(K^\alpha)$ be the corresponding
Lebesgue constant. Then
$$\Lambda_{N+1}(K^\alpha) \leq h_0^{-N}\cdot (N+1).$$
\end{proposition}
\begin{proof} It is evident in view of \eqref{int} and Lemma \ref{max}.
\end{proof}

From now on, for any product $\prod_{k=1}^M t_k$ with $t_k\geq 0$ and $p<M,$ we will use the symbol $ _p(\prod_{k=1}^M t_k)$ to denote the product of $p$ smallest terms $t_k$. Also, let $(\prod_{k=1}^M t_k)_p$ be the original product without $p$ smallest terms, so
$\prod_{k=1}^M t_k= _p(\prod_{k=1}^M t_k)\cdot (\prod_{k=1}^M t_k)_p.$

Of course, in general, $\prod_{k=1}^M \tau_k \leq \prod_{k=1}^M t_k$ does not imply

\begin{equation}\label{tau}
\left(\prod_{k=1}^M\tau_k \right)_p \leq \left(\prod_{k=1}^M t_k\right)_p.
\end{equation}
But \eqref{tau} is trivially valid provided additional condition:
$0\leq \tau_k\leq t_k$ for all $k$.\\

Let us show that in \eqref{mu-nu} and \eqref{l1} $\nu_j$ can be replaced by $\lambda_j$.
The argument of Lemma \ref{max} can be applied to each point $x\in K^\alpha$ instead of $\tilde x.$ Let us apply it
to $x_{N+2},$ for which we have some degrees $\overline{\nu}_j$ and  \eqref{mu-nu} with $\overline{\nu}_j$ instead of $\nu_j.$
If we add the point $x_k$ to $Z\setminus\{x_k\}$ then one of $\overline{\nu}_j$ will increase by one and new powers become $\lambda_j$. Thus,
\begin{equation}\label{la-mu}
\mu_s+\mu_{s-1}+\cdots +\mu_i\leq \lambda_s+\lambda_{s-1}+\cdots +\lambda_i
\end{equation}
for $1\leq i\leq s.$ By this, as above, we have $\prod_{j=0}^s \ell_j^{\lambda_j} \leq \prod_{j=0}^s  \ell_j^{\mu_j}.$  Note that the values of
$\mu_j, \overline{\nu}_j$ depend on $k$, while $\lambda_j$ is only defined by $N$. The left product has $N+1$ terms, and the right product has $N$.
Of course, we can start multiplication starting from $j=1$. By \eqref{la-mu}, it can be started from each $i$. In fact, we can make a more general estimate.
Let $p\leq N.$ Then
\begin{equation}\label{p-l}
_p\left(\prod_{j=0}^s \ell_j^{\lambda_j}\right) \,\leq\,\,\,  _p\left(\prod_{j=0}^s \ell_j^{\mu_j}\right).
\end{equation}

The proof is by induction on $p$. Let $p=1.$ The smallest term of the $\lambda-$product is $\ell_s$ as for each $N$ with \eqref{NN}, the value of
$\lambda_s$ is equal to one. On the other hand, $\mu_s=m_{j,s}-1$ with $m_{j,s}\in\{1,2\}$ as was discussed before Lemma \ref{max}.
Hence, the smallest term of the $\mu-$product is $\ell_s$ or $\ell_{s-1}$ and \eqref{p-l} is valid. Suppose it is true for some $p\geq 1.$
Let $p=\lambda_s+\cdots +\lambda_i+\tau$ with $0\leq \tau<\lambda_{i-1}.$ Then
$_{(p+1)}\left(\prod_{j=0}^s \ell_j^{\lambda_j}\right)=_p\left(\prod_{j=0}^s \ell_j^{\lambda_j}\right)\cdot \ell_{i-1}.$
In its turn, $_{(p+1)}\left(\prod_{j=0}^s \ell_j^{\mu_j}\right)=$ $_p\left(\prod_{j=0}^s \ell_j^{\mu_j}\right)\cdot t.$ By \eqref{la-mu},
$\mu_s+\mu_{s-1}+\cdots +\mu_i\leq p.$ Consequently, the product $_p\left(\prod_{j=0}^s \ell_j^{\mu_j}\right)$ consists of the corresponding powers of
$\ell_s, \ldots, \ell_i$, possibly with some number of larger terms. Therefore, for $p+1-$st, the term $t$ cannot be less than $\ell_{i-1}$
and \eqref{p-l} is valid for each $p$.\\

The $(\cdot)_p-$version of \eqref{p-l}
\begin{equation}\label{l-p}
\left(\prod_{j=0}^s \ell_j^{\lambda_j}\right)_p \,\leq\,\,\,  \left(\prod_{j=0}^s \ell_j^{\mu_j}\right)_p
\end{equation}
is also correct. Let us first show a stronger result.

\begin{lemma}\label{4.2}
Given $N\in {\Bbb N},$ let $Z=(x_k)_{k=1}^{N+1}$ be chosen in $K^\alpha$ by the rule of increase of type. Suppose $1\leq k\leq N+1$ and
$(\mu_j)_{j=0}^s$ are degrees corresponding to $x_k$. For a fixed $x\in K^{\alpha}$, let $(\nu_j)_{j=0}^s$ be defined as in Lemma \ref{max} with
$x$ instead of $\tilde x.$ Then for $ 1\leq p\leq N$ we have
\begin{equation}\label{nu-p}
\left(\prod_{j=0}^s \ell_j^{\nu_j}\right)_p \,\leq\,\left(\prod_{j=0}^s \ell_j^{\mu_j}\right)_p.
\end{equation}
\end{lemma}
\begin{proof} Let $\sigma_q(\mu)=\mu_s+\mu_{s-1}+\cdots +\mu_q$ for $0\leq q\leq s$ with a similar definition of $\sigma_q(\nu).$
As in Lemma \ref{max}, $I_{j_{s-n},n}=\tilde I_{j_{s-n},n}$ is the smallest basic interval containing both points $x_k$ and $x.$
Then $\sigma_n(\mu)=\sigma_n(\nu)$ as $\mu_i=\nu_i$ for $0\leq i\leq n-1.$ It follows that \eqref{nu-p} is the equality for $\sigma_n(\mu)\leq p\leq N.$
In particular,
$\left(\prod_{j=0}^s \ell_j^{\mu_j}\right)_{\sigma_n(\mu)}= \ell_{n-1}^{\mu_{n-1}}\cdots \ell_0^{\mu_0}=\ell_{n-1}^{\nu_{n-1}}\cdots \ell_0^{\nu_0}.$
For brevity, we denote this product by $A$.

By decreasing induction on $p$, suppose that $\sigma_{n+1}(\mu)\leq p <\sigma_n(\mu)$ so $p=\sigma_{n+1}(\mu)+\tau$ with $0\leq \tau<\mu_n.$
Removing the smallest $p$ terms from $\prod_{j=0}^s \ell_j^{\mu_j}$ gives $\left(\prod_{j=0}^s \ell_j^{\mu_j}\right)_p=\ell_n^{\mu_n-\tau}\cdot A.$
It follows that to get $\left(\prod_{j=0}^s \ell_j^{\nu_j}\right)_p$
we must multiply the product $A$ by the $\mu_n-\tau$ largest terms from  $\ell_s^{\nu_s} \cdots \ell_n^{\nu_n}.$ Since these terms do not exceed $\ell_n,$
we have \eqref{nu-p} for a given $p.$

Suppose, \eqref{nu-p} is valid with the subscript $\sigma_q(\mu)$ for $ n+1<q<s.$ Let $p=\sigma_{q+1}(\mu)+\tau$ with $0\leq \tau<\mu_q.$
Then $(\prod_{j=0}^s \ell_j^{\mu_j})_p=\ell_q^{\mu_q-\tau}\cdot (\prod_{j=0}^s \ell_j^{\mu_j})_{\sigma_{q}(\mu)}.$
On the other hand, $(\prod_{j=0}^s \ell_j^{\nu_j})_p=t_1\cdots t_{\mu_q-\tau}\cdot (\prod_{j=0}^s \ell_j^{\nu_j})_{\sigma_{q}(\mu)},$
where $(\prod_{j=0}^s \ell_j^{\nu_j})_{\sigma_{q}(\mu)}\leq (\prod_{j=0}^s \ell_j^{\mu_j})_{\sigma_{q}(\mu)}$ by  the induction hypothesis and
$(t_i)_{i=1}^{\mu_q-\tau}$ are the next descending members of $\ell_s^{\nu_s} \cdots \ell_0^{\nu_0}$ after the largest $\sigma_{q}(\mu)$ terms
have been removed. By \eqref{mu-nu}, these $t_i$ are among $\ell_s^{\nu_s} \cdots \ell_q^{\nu_q},$ which completes the proof.
\end{proof}

{\bf Remark}. Multiplying the left side of \eqref{nu-p} by some additional term $\ell_j$ can only reduce it. This implies the inequality \eqref{l-p}.\\

Now the task is to find an analog of the Markov property for the sets under consideration.

Let $K\subset \Bbb R$ be a compact set  of infinite cardinality. A sequence of {\it Markov's factors} for $K$ is defined as
 $M_N(K)=\inf \{M: \,|Q'|_K \leq M \,|Q|_K, \,\,Q\in {\mathcal P}_N \}$
for  $N\in {\Bbb N}.$ Thus, there are constants $C$ and $r$ with $M_N(K)\leq C N^r$ for all $N$ under the Markov property of $K$.
We see that $M_N(K)$ is the norm of the differentiation operator $D$ in the space $({\mathcal P}_N, |\cdot|_K).$

Given $p, N\in {\Bbb N},$ we define Markov's $N-$th factor of $p-$th order as the norm of $D^p$:
$$M_N^{(p)}(K)=\inf \{M: \,|Q^{(p)}|_K \leq M \,|Q|_K, \,\,Q\in {\mathcal P}_N \}.$$

Clearly, $M_N^{(p)}(K)\leq M_N(K)^p.$ This estimate is not rough for Markov sets. For example, if $K=[-1,1]$ then, see, e.g., \cite{VL}, p. 132,
$M_N^{(p)}(K)=\frac{N^2\cdot(N^2-1) \cdots (N^2-(p-1)^2)}{1\cdot 3\cdot \cdots (2p-1)}$ with $M_N(K)=N^2.$
For the Cantor sets under consideration, the difference between $M_N^{(p)}(K)$ and $M_N(K)^p$  is essential.

In the next lemma, given $N\in {\Bbb N}$, we use \eqref{NN} and \eqref{pi} for $N+1, \rho_1\cdots \rho_p=_p\left(\prod_{j=0}^s \ell_j^{\lambda_j}\right).$

\begin{theorem}  \label{mf}
Given $N$ and $1\leq p<N,$ we have $M_N^{(p)}(K^{\alpha})\leq h_0^{-N}\frac{(N+1) N^p}{\rho_1\cdots \rho_p}.$
\end{theorem}
\begin{proof} Fix $Q\in {\mathcal P}_N$. There is no loss of generality in assuming $|Q|_{K^{\alpha}}=1.$
In view of \eqref{int}, it suffices to show that
 \begin{equation}\label{apk}
 \frac{|a^{(p)}_k(x)|}{|a_k(x_k)|}  \leq N^p\,h_0^{-N}\, (\rho_1\cdots \rho_p)^{-1}\,\, \text {for}\,\,\, 1\leq k\leq N+1,\, x\in K^{\alpha}.
 \end{equation}
Fix $k$ and $x$. Let degrees $(\nu_j)_{j=0}^s$ correspond to $x$, as it was in Lemma \ref{max} for $\tilde{x}$. Then
$|a_k(x)|=\prod_{i=1}^N  d_i\leq \prod_{j=0}^s \ell_j^{\nu_j}$ with $d_i:=d_i(x,Z\setminus\{x_k\}).$ Suppose $\nu_s\geq 1.$ Then $0\leq d_1\leq \ell_s.$
Due to the choice of the degrees, for every $2\leq i\leq N$ there exists $n(i)$ with  $h_{n(i)}\leq d_i\leq \ell_{n(i)}.$
If  $\nu_s=0$ this is also true for $i=1$ with $n(1)=s-1.$ Here $x$ and $x_k$ belong to the same $I_{j,s}$ with $m_{j,s}=1.$
Since $p\geq 1$, by \eqref{tau}, $(\prod_{j=0}^s h_j^{\nu_j})_p \leq (\prod_{i=1}^N  d_i)_p \leq (\prod_{j=0}^s \ell_j^{\nu_j})_p.$
By \eqref{om}, $|a^{(p)}_k(x)|\leq N^p (\prod_{j=0}^s \ell_j^{\nu_j})_p.$ Applying Lemma \ref{4.2} and \eqref{pi2} yields
$$
\frac{|a^{(p)}_k(x)|}{|a_k(x_k)|}  \leq \frac{N^p (\prod_{j=0}^s \ell_j^{\mu_j})_p }{h_0^{N} \prod_{j=0}^s \ell_j^{\mu_j}}=
\frac{N^p}{h_0^{N}\, _p(\prod_{j=0}^s \ell_j^{\mu_j})},
$$
which gives \eqref{apk}, by \eqref{p-l}.
\end{proof}

The inequality in the previous proposition is exact with respect to the terms $\rho_i.$

{\bf Example.} Let $N=2^s.$ Consider a polynomial $\omega_N(x)=\prod_{k=1}^N(x-x_k)$ that has zeros at all points from $Y_{s-1}$. Then $|\omega_N|_{K^{\alpha}}=|\omega_N(\ell_{s})|\leq \ell_{s}\,\ell_{s-1}\,\ell_{s-2}^2 \cdots \ell_0^{2^{s-1}}$
$=\prod_{i=1}^N \rho_i.$ The exact value of $|\omega_N|_0$  is $\ell_{s} \prod_{i=2}^N d_i(x)$ with $d_i(x)=d_i(x,(x_k)_{k=1}^N)$. As above,
$h_0 \rho_i\leq  d_i(x)\leq \rho_i$ for $i\geq 2.$ Then $|\omega^{(p)}_N(0)|\geq \prod_{i=p+1}^N d_i(0),$ because $\omega^{(p)}_N(0)$ is a sum of
products of the same sign and one of them is $\prod_{i=p+1}^N d_i(0).$ Consequently,
$M_N^{(p)}(K^{\alpha})\geq \frac{|\omega^{(p)}_N(0)|}{|\omega_N|_{K^{\alpha}}}\geq h_0^{N-p}(\rho_1\cdots \rho_p)^{-1}.$

\section{Jackson's type inequality by means of Faber bases}

Suppose that $X(K)$ is a space of functions on $K$, containing polynomials. A polynomial topological basis $(Q_n)_{n=0}^{\infty}$ in $X$ is called a Faber
(or strict polynomial) basis if $ \mathrm{deg}\, Q_n =n$ for all $n$. Thus, for each $f\in X$ there is a unique number sequence  $(\xi_n(f))_{n=1}^{\infty}$
such that the series $\sum_{n=0}^{\infty}\,\xi_n(f)\,Q_n$ converges to $f$ in the topology of $X.$ This gives an estimate of the best uniform approximation
of $f$ by polynomials:
\begin{equation}\label{EN}
E_N(f,K)\leq \sup_{x\in K}|\sum_{n=N+1}^{\infty}\,\xi_n(f)\,Q_n(x)|.
\end{equation}

{\bf Example.} By Lemma 25 in \cite{mi}, the Chebyshev polynomials $(T_n)_{n=0}^{\infty}$ form a basis in the space $C^{\infty}[-1,1]$. The corresponding
biorthogonal functionals are given as follows $\xi_0(f)= \frac{1}{\pi} \int_0^{\pi} f(\cos t)dt,$ $ \xi_n(f) = \frac{2}{\pi} \int_0^{\pi} f(\cos t) \cos nt dt, $
$ n \in \Bbb N.$  By (44) in \cite{mi} (Jackson's theorem is not used!), $|\xi_n(f)|\leq \frac{C_k\,|f|_k}{n^k}$ for each $k$.
Since the basis is absolute and $|T_n|_0=1,$ we get $E_N(f,[-1,1])\leq \sum_{n=N+1}^{\infty} |\xi_n(f)|$. This gives $E_N(f,[-1,1])\leq \frac{C_q\,|f|_{q+1}}{N^q}$ for $q\in \Bbb N.$\\

In the case of small sets $K$, the phenomenon of ultra-fast convergence of polynomials to functions from ${\mathcal E}(K)$ is observed.
Let $\omega_0=1$ and $\omega_n(x)=\prod_{k=1}^n(x-x_k)$ for $n\in {\Bbb N},$ where the points $(x_k)_{k=1}^n$ are chosen in $K^{\alpha}$ by the rule of increase of type. Given $f\in X(K)$ and $n\in{\mathbb Z}_+,$ by $\xi_n(f)$ we denote the divided difference $[x_1,x_2,\cdots,x_{n+1}]f.$ The functionals
$(\xi_n)_{n=0}^{\infty}$ are biorthogonal to $(\omega_n)_{n=0}^{\infty}$. If $\alpha\geq 2$ then, by Theorem 1 in \cite{CA},  the sequence $(\omega_n)_{n=0}^{\infty}$ is a basis in the space ${\mathcal E}(K^{\alpha}).$ This allows us to evaluate $E_N(f,K^{\alpha})$ in terms of the values
$(\rho_i)_{i=1}^{N+1}$ determined in \eqref{pi}. In order to do this, we define analogous $\rho_k(n)$ for another $n.$

Let $2^r\leq n<2^{r+1}$. Then $n=2^r+2^{r_1}+ \cdots +2^{r_m}$ with $0\leq r_m < \cdots < r_1 < r_0:=r.$ In the same way as $\pi_j$ for $N+1$,
we define $\pi_j(n)$ and $(\lambda_j(n))_{j=0}^r$ so that $\prod_{j=0}^m\pi_j(n)=\prod_{i=0}^r  \ell_i^{\lambda_i(n)}=\prod_{k=1}^n \rho_k(n)$
with nondecreasing $\rho_k(n).$ Let us point out some obvious properties of the degrees $\lambda_j(n).$  First, $\lambda_j(n)\leq \lambda_j(n+1)$ with $\lambda_j(n)=\lambda_j(n+1)$ for all $j$ except some $j_0$ for which $\lambda_{j_0}(n+1)=\lambda_{j_0}(n)+1.$ Secondly,
$\prod_{i=0}^r  \ell_i^{\lambda_i(2^r)}=\pi_0(2^r)=\ell_{r}\,\ell_{r-1}\,\ell_{r-2}^2 \cdots \ell_0^{2^{r-1}},$ whereas
$\prod_{i=0}^r  \ell_i^{\lambda_i(2^{r+1}-1)}=\ell_{r}\,\ell_{r-1}^2\,\ell_{r-2}^4 \cdots \ell_0^{2^{r}}.$

Therefore, if $n$ is as above, then $\lambda_j(2^r)\leq \lambda_j(n)\leq \lambda_j(2^{r+1}-1)$ with $\lambda_r(n)=1$ and
$2^{r-j-1}\leq \lambda_j(n)\leq 2^{r-j}$ for $0\leq j \leq r-1.$ If $q=2^w<n$ then
\begin{equation}\label{ql}
_q\left(\prod_{j=0}^r \ell_j^{\lambda_j(n)}\right) \,\leq\,_q\left(\prod_{j=0}^r \ell_j^{\lambda_j(2^r)}\right) = \ell_{r}\,\ell_{r-1}\,\cdots \, \ell_{r-w}^{2^{w-1}}.
\end{equation}

As in Proposition \ref{mf}, for a given $N$ with $2^s\leq N+1<2^{s+1}$, we use \eqref{NN} and \eqref{pi} for $N+1$
and the corresponding $(\rho_i)_{i=1}^{N+1}$ with $\rho_i=\rho_i(N+1).$

\begin{theorem}\label{JT}
Suppose $\alpha\geq 2$ and $\ell_1\leq1/4$. Let $N$ be as above. Then for each $f\in {\mathcal E}(K^{\alpha})$ and $q=2^w$ with $w<s-8$
we have $E_N(f,K^{\alpha})\leq C_q\,\rho_1\cdots \rho_q \, ||f||_{q_1},$ where $C_q$ does not depend on $f$ and $N, q_1=2^{w+8}+1.$
\end{theorem}
\begin{proof} Fix $N, q$ and $f$ as above. By \eqref{EN}, $E_N(f,K^{\alpha})\leq \sum_{n=N+1}^{\infty}\,|\xi_n(f)|\cdot |\omega_n|_0$
with the decomposition $\sum_{n=N+1}^{\infty}=\sum_{n=N+1}^{2^{s+1}-1}+ \sum_{r=s+1}^{\infty}\sum_{n=2^r}^{2^{r+1}-1}=:\Sigma_1+\Sigma_2.$

Let $2^r\leq n<2^{r+1}$ with $r\geq s+1.$ To estimate $|\xi_n(f)|\cdot |\omega_n|_0$ from above, we use the arguments of Theorem 1 in \cite{CA} with minor modifications. For each $x\in K^{\alpha},$ the estimate $|\omega_n(x)|\leq \prod_{j=0}^r \ell_j^{\lambda_j(n)}$ holds true.
We apply (2) in \cite{CA} with $q_1$ instead of $q$ using the following two improvements. The Open Mapping Theorem can be applied to the space
${\mathcal E}^{q_1}(K^{\alpha}),$  giving the $q_1$ version of \eqref{quot} instead of (3) in \cite{CA}. And three lines above (3) in \cite{CA}
we do not replace $1-2\ell_1=h_0$ by $\ell_1.$ This gives
$|\xi_n(f)|\leq C ||\,f\,||_{\,q_1} 2^n\cdot  h_0^{-n} (\prod_{j=0}^r \ell_j^{\lambda_j(n)})^{-1}_{q_1}.$ Hence,
$$ |\xi_n(f)| \cdot |\omega_n|_0\leq C\,||\,f\,||_{\,q_1} \,\left(\frac{2}{h_0}\right)^n\,_{q_1}\left(\prod_{j=0}^r \ell_j^{\lambda_j(n)}\right).$$

Let's show that $\sum_{n=2^r}^{2^{r+1}-1} \left(\frac{2}{h_0}\right)^n\,_{q_1}\left(\prod_{j=0}^r \ell_j^{\lambda_j(n)}\right) \leq \,_{q}\left(\prod_{j=0}^r \ell_j^{\lambda_j(2^r)}\right)$. By condition, $\frac{2}{h_0} \leq \frac{1}{\ell_1},$ so $\left(\frac{2}{h_0}\right)^n< \ell_1^{-2^{r+1}}.$
On the other hand, the second term of the product in the sum above takes maximum value if $n=2^r.$ Therefore the whole sum does not exceed
$2^r \ell_1^{-2^{r+1}}\, _{q_1}\left(\prod_{j=0}^r \ell_j^{\lambda_j(2^r)}\right).$ Here the last term is
$ \,_{q}\left(\prod_{j=0}^r \ell_j^{\lambda_j(2^r)}\right)\cdot \{\ell_{r-w-1}^{2^{w}}\cdots \ell_{r-w-8}^{2^{w+7}}\cdot \ell_{r-w-8}\}.$
Of course, $2^r\ell_{r-w-8}<1$ for sufficiently large $r$. Also we have in braces 8 terms of the type $\ell_{r-k-1}^{2^{k}}$.
Since $\alpha\geq 2$, for each of them we have  $\ell_{r-k-1}^{2^{k}}=\ell_1^{\alpha^{r-k-2} 2^k}\leq \ell_1^{2^{r-2}}$, so their product neutralizes $\ell_1^{-2^{r+1}}$. From this
$$ \Sigma_2\leq C\,||\,f\,||_{\,q_1} \sum_{r=s+1}^{\infty} \,_{q}\left(\prod_{j=0}^r \ell_j^{\lambda_j(2^r)}\right).$$

It is easy to check that the first term in the above sum dominates, so the whole sum does not exceed twice the first term. By monotonicity,
$\lambda_j( 2^{s+1})\geq \lambda_j(N+1)$ for all $j$. Therefore,
$$ \Sigma_2\leq 2C\,||\,f\,||_{\,q_1}\,\cdot _{q}\left(\prod_{j=0}^s \ell_j^{\lambda_j(N+1)}\right)= 2C\,||\,f\,||_{\,q_1}\,\rho_1\cdots \rho_q .$$

Similar arguments apply to $\Sigma_1$ with $N+1\leq n <2^{s+1}$, but now we estimate $|\xi_n(f)|\cdot |\omega_n|_0$ directly using
$_{q}\left(\prod_{j=0}^r \ell_j^{\lambda_j(N+1)}\right).$ This gives the desired result.
\end{proof}

{\bf Remark.} The condition $\ell_1\leq 1/4$ is not particularly restrictive. Enlarging $q_1$ allows us to neutralize $\left(\frac{2}{h_0}\right)^n$
for larger values of $\ell_1$ as well. However, the condition $\alpha\geq 2$ is important here since the sequence $(\omega_n)_{n=0}^{\infty}$
is not a basis in ${\mathcal E}(K^{\alpha})$ for $\alpha < 2.$ We believe that these spaces also have Faber interpolation bases with a different,
more complex choice of interpolation nodes, but we cannot present them. For this reason, our main result is given only for ${\mathcal E}(K^2).$

Comparison of Theorems \ref{mf} and \ref{JT}  shows that $E_N(f,K^{\alpha})$ successfully neutralizes the fast growth of factors $M_N^{(p)}(K^{\alpha})$
for all $\alpha\geq 2$. Let us show that an even stronger fact holds for $\alpha=2$.

\begin{proposition}\label{mf-jt}
For each $p$ and $r>p$ there exists $r_1$ such that
$$M_N^{(p)}(K^2)\cdot E_{N-1}(f,K^2)\leq \rho_{p+1}(N+1)\cdots \rho_r(N+1)\,||\,f\,||_{\,r_1}$$
for sufficiently large $N$. Here we assume $\ell_1\leq 1/3.$
\end{proposition}
\begin{proof} As above, $2^s\leq N+1<2^{s+1}.$ Recall that $M_N^{(p)}(K^2)$ is determined by $\rho_k(N+1),$ whereas $E_{N-1}(f,K^2)$
is given in terms of $\rho_k(N).$ Fix $p$ and $r>p.$ There is no loss of generality in assuming $r=2^w$. Set $r_1=2^{w+10}.$ We apply Theorem \ref{JT}
with $r$ instead of $q$. We need to show
$$ C_r (N+1)N^p h_0^{-N} \frac{\rho_1(N)\cdots \rho_{r_1}(N)}{\rho_1(N+1)\cdots \rho_p(N+1)}\leq \rho_{p+1}(N+1)\cdots \rho_r(N+1)$$
or equivalently
$$ C_r (N+1)N^p h_0^{-N} \frac{\rho_1(N)\cdots \rho_{r-1}(N)}{\rho_1(N+1)\cdots \rho_r(N+1)}  \rho_r(N)\cdots \rho_{r_1}(N)\leq 1.$$
The fraction in the middle is $\frac{_{r-1}\left(\prod_{j=0}^s \ell_j^{\lambda_j(N)}\right)}{_r\left(\prod_{j=0}^s \ell_j^{\lambda_j(N+1)}\right)},$
where the denominator contains all the factors of the numerator, so the fraction is $\frac{1}{\ell_j}$ with some $0\leq j \leq s.$ In the worst case
it is $\frac{1}{\ell_s}.$ In addition, $h_0\geq \ell_1$ and $N<2^{s+1}$ imply $h_0^{-N}< \ell_1^{-2^{s+1}}=\ell_s^{-4}$. It suffices to prove that
\begin{equation}\label{mj}
C_r (N+1)N^p \rho_r(N)\cdots \rho_{r_1}(N)\leq \ell_s^5.
\end{equation}
Since $N\geq 2^s-1,$ we have $\prod_{j=0}^N\rho_j(N)\leq \prod_{j=0}^N\rho_j( 2^s-1)= (\ell_{s}\,\ell_{s-1}^2\, \cdots \ell_{s-w}^{2^{w-1}})\cdot \ell_{s-w-1}^{2^{w}}\cdot\ell_0^{2^{s}},$ where the product in parentheses contains $r-1$ terms.  We can only enlarge the left side of \eqref{mj} by replacing $\rho_r(N)\cdots \rho_{r_1}(N)$ with $\ell_{s-w-1}^{2^{w}}\cdots\ell_{s-w-10}^{2^{w+9}}\cdot \ell_{s-w-11},$ containing exactly $r_1-r+1$ terms.
Now the product $C_r (N+1)N^p\cdot \ell_{s-w-11}$ does not exceed 1 for sufficiently large $N$. Also, $\ell_{s-k-1}^{2^k}=\ell_1^{2^{s-2}},$ and we have
$10$ such terms, so their product is $\ell_1^{10\cdot 2^{s-2}},$ which is equal to $\ell_s^5.$ This is the desired conclusion.
\end{proof}

\section{Simultaneous extensions of basic polynomials}
How does $K^2$ with the extension property differ from $K^{\alpha}, \alpha > 2,$ without it?
Let us show that the difference depends mainly on the possibility of suitable individual extensions of the basic interpolation polynomials.

Let $K$ and  $Q\in {\mathcal P}_N$ be such as in Section 2. We fix $\delta>0$ and a segment  $I$ containing $K_{\delta}$.
Write $\tilde Q=Q\cdot u_{\delta}.$ Clearly, $|Q|_{0,K}\leq |\tilde Q|_{0,I}.$  We will use the notation $\tilde Q\sim Q$ if
$|\tilde Q|_{0,I}\leq C |Q|_{0,K}$ with some $C$ independent of $Q$ and $N$. If $K$ is Markov, then by \eqref{IE} the choice $\delta=N^{-r}$
provides $|\tilde Q|_{0,I}\leq C |Q|_{0,K}$. In addition, suppose that $Q$ realizes $M_N^{(p)}(K)$ for $p<N.$
This means that the converse of \eqref{MP} is true with some constant $c_1$ instead of $C_1$, similar to how Chebyshev polynomials implement $M_N^{(p)}([-1,1])$.
Then by \eqref{IE1},  $\tilde Q^{(p)}\sim Q^{(p)},$ so there exists $\delta=\delta(N)$ which can be applied to extend both $Q$ and its derivatives.

In general, $\delta$ depends on $p$. Let us illustrate this with an apparent example. From now on, $u_{\delta}$ is determined by the function
$\varphi$ with $\varphi(x)=1$ for $x\leq 0,\,\varphi(x)=0$ for $x\geq 1$ and
\begin{equation}\label{FF}
\varphi(x)=\exp\left[\frac{1}{x-1} \exp\left(-\frac{1}{x}\right)\right]
\end{equation}
for $0<x<1.$ We see that $\varphi\in C^{\infty}(\Bbb R)$ with $\varphi(\frac{1}{2})>\frac{1}{2}.$ Given $\delta>0,$ let
$\varphi_{\delta}(x)=\varphi(x/\delta).$ Given $K$ with a complementary interval $(a,b)$ with $b-a\geq 2\delta$, we define $u_{\delta}=1$ on $K$,
$u_{\delta}(x)=\varphi_{\delta}(x-a)$ for $a<x<a+\delta, u_{\delta}(x)=b-x$ for $b-\delta<x<b$, and $u_{\delta}(x)=0$ for $a+\delta\leq x\leq b-\delta.$
If $b-a<2\delta$ then $u_{\delta}=1$ on $(a,b).$ Then $u_{\delta}$ has the desired properties indicated in Section 2.

{\bf Example.} Let $K=[-\varepsilon, \varepsilon]$ for a small $\varepsilon.$ Of course, $K$ is Markov, but with the constant $C_1$ in \eqref{MP}
depending essentially on $\varepsilon.$ Indeed, fix $\delta>0$. As above, $\tilde Q=Q\cdot u_{\delta}.$ Suppose that $|\tilde Q|_{0,I}\leq C |Q|_{0,K}$ and $|\tilde Q''|_{0,I}\leq C n^2 |Q|_{2,K}$ hold for $Q \in {\mathcal P}_n,$ where $I=[-1,1].$ Then for $Q(x)=x$ we have
$|\tilde Q|_{0,I}\geq \tilde Q(\varepsilon+\frac{\delta}{2})>(\varepsilon+\frac{\delta}{2})\frac{1}{2}>\frac{\delta}{4}$ as
$\varphi_{\delta}(\frac{\delta}{2})>\frac{1}{2}.$ On the other hand, $Q'(\varepsilon)=1, Q'(\varepsilon+\delta)=0,$ so, by the mean value theorem,
$|\tilde Q''|_{0,I}\geq \frac{1}{\delta}.$ Here, $|Q|_{0,K}=\varepsilon, |Q|_{2,K}=1.$ Hence, if the inequalities above hold then
$\frac{\delta}{4}\leq C \varepsilon$ and $\frac{1}{\delta}\leq C,$ which means that $C\geq \frac{1}{2\sqrt{\varepsilon}}.$\\

An extension of polynomials that provides the above equivalence can be called {\it strong simultaneous extensions}. The term {\it simultaneous extensions}
was used in \cite{og} for the existence of a linear extension operator. In our case, instead of \eqref{IE}, we consider the following condition
for the sequence $(\omega_n)_{n=0}^{\infty}$ defined in Section 5
\begin{equation}\label{seom}
\exists (\delta_n)_{n=1}^{\infty}: \forall p \exists q, C: \,\,\,\,\,|\tilde{\omega}_n|_p\leq C |\omega_n|_q, \,\,\,n\in \Bbb N.
\end{equation}

\begin{proposition}\label{=2}
Choosing $\delta_n=\ell_s$ for $2^s\leq n<2^{s+1}$ provides \eqref{seom} on $K^2$. Here we assume $\ell_1\leq 1/3.$
\end{proposition}
\begin{proof}
The previous formulas are given mainly for $N+1$, so we prove \eqref{seom} for $n=N+1$. Both parts of \eqref{seom} will be expressed in terms
of $\rho_k(N+1)$ with $N$ as in \eqref{NN}. Thus, $\delta_{N+1}=\ell_s.$
Since $\delta_{N+1}$ is not included in the right-hand side of \eqref{seom},  we first estimate $||\omega_{N+1}||_q$ for a given $1\leq q<N+1.$
Henceforth $q=2^w+1$, where $w$ will be defined later depending on $p$. Our claim is that
\begin{equation}\label{qqq}
|\omega_{N+1}^{(q)}|\geq h_0^{N+1-q}\cdot \rho_q \cdots \rho_{N+1}.
\end{equation}
Recall that $\omega_{N+1}(x)=\prod_{i=1}^{N+1}(x-x_i),$ where $Z=(x_k)_{k=1}^{N+1}$ are chosen by the rule of increase of type, so $x_1=0$. Then
$a_1(x)=\prod_{i=2}^{N+1}(x-x_i).$ Let $d_k:=d_k(0,Z)$. Then, of course, $d_1=0$ and $d_{k+1}=d_k(0,Z\setminus\{0\})$ for $1\leq k<N.$
The function $a_1$ determines the powers $(\mu_j)_{j=0}^s$ such that \eqref{pi2} holds for $|a_1(0)|=\prod_{k=1}^{N}d_k(0,Z\setminus\{0\}).$
As in Theorem \ref{mf}, for each $1\leq k<N$ there is an index $j$ with $h_j\leq d_k(0,Z\setminus\{0\})\leq \ell_j.$
Removing the $q-1$ smallest terms from the above product gives
$\prod_{k=q}^{N}d_k(0,Z\setminus\{0\})\geq \left(\prod_{j=0}^s h_j^{\mu_j}\right)_{q-1}\geq h_0^{N+1-q}\left(\prod_{j=0}^s \ell_j^{\mu_j}\right)_{q-1}.$

We now turn to $\omega_{N+1}^{(q)}$. It is a sum of $\frac{(N+1)!}{(N+1-q)!}$ products, each containing $N+1-q$ terms, so they all have the same sign,
and one of them is $d_{q+1}\cdots d_{N+1}.$ Therefore,
$$ |\omega_{N+1}^{(q)}|\geq \prod_{k=q+1}^{N+1}d_k= \prod_{k=q}^{N}d_k(0,Z\setminus\{0\})\geq  h_0^{N+1-q}\left(\prod_{j=0}^s \ell_j^{\mu_j}\right)_{q-1}.$$

Applying \eqref{l-p} yields \eqref{qqq}.\\

We proceed to estimate from above $|\tilde{\omega}_{N+1}^{(p)}(x)|$ for fixed $p$ and $x$. It is clear that only $x$ outside the set should be considered.
Fix $x$ with $0<\mathrm{dist}(x,K^2)=|x-y|<\ell_s.$
By \eqref{pi}, $|\omega_{N+1}(y)|\leq \prod_{k=1}^{N+1}\rho_k,$ so $|\omega_{N+1}(x)|\leq \prod_{k=1}^{N+1}(\rho_k+\ell_s).$
Here, $\rho_1=\ell_s, \rho_2=\ell_{s-1}$ and $\rho_3=\ell_{s-1}$ or $\ell_{s-2},$ depending on the value of $N+1.$
In a fairly straightforward way, one can show that $\prod_{k=2}^{N+1}(\rho_k+\ell_s) \leq 2 \prod_{k=2}^{N+1}\rho_k.$ This and \eqref{om} give
$|\omega_{N+1}^{(j)}(x)|\leq 4 (N+1)^j \,\prod_{k=j+1}^{N+1}\rho_k.$ Also, $|u^{(i)}_{\delta_{N+1}}(x)|\leq c_i \ell_s^{-i}$.
Without loss of generality, we can assume that $c_i$ increases with $i$. Then, by Leibnitz's rule,
$$|(\omega_{N+1}\cdot u_{\delta})^{(p)}(x)|\leq 4 c_p \sum_{j=0}^p  \binom{p}{j}(N+1)^j \,t_j,$$
where  $t_j:=\ell_s^{-p+j}\,\prod_{k=j+1}^{N+1}\rho_k.$ We see that $t_0=t_1=\ell_s^{-p+1}\cdot \rho_2 \cdots \rho_{N+1},$ while the following terms decrease very rapidly. Hence, $|(\omega_{N+1}\cdot u_{\delta})^{(p)}(x)|\leq 4 c_p t_0 [1+p(N+1)+o(1)]\leq c'_pNt_0,$
where $c'_p$ does not depend on $N$. Combining this with \eqref{qqq} we reduce the desired inequality to
$N\cdot \rho_2 \cdots \rho_{q-1} \leq  C\ell_1^{N+1-q}\cdot \ell_s^{p-1},$ because $\ell_1\leq h_0$ by condition.
Let us replace $\rho_k$ with $\rho'_k$ determined by $2^s.$ Then
$\prod_{k=2}^{q-1}\rho'_k=\ell_{s-1}\ell_{s-2}^2\cdots \ell_{s-w}^{2^{w-1}}\geq \prod_{k=2}^{q-1}\rho_k.$ It remains to prove that
\begin{equation}\label{last}
N\cdot \ell_{s-1}\ell_{s-2}^2\cdots \ell_{s-w}^{2^{w-1}}\leq C \ell_1^{N+1-q}\cdot \ell_s^{p-1}.
\end{equation}
In the case under consideration, we have $\ell_k=\ell_1^{2^{k-1}}.$ Therefore, the left side of \eqref{last} is $N\,\ell_1^{\,w 2^{s-2}}$ and
$C \ell_1^{N+1-q+(p-1)2^{s-1}}$ is on the right. Since $N+1<2^{s+1},$ we reduce \eqref{last} to $2^{s+1} \ell_1^{2^{s-2}(w -2p-6+q)}\leq C,$
which holds for $s\geq 3$ provided $w=2p+6.$

Small values of $s$ do not cause problems, since $|\omega_{N+1}^{(q)}|=(N+1)!$ for such $s$ and given $q$ and $|\tilde{\omega}_n|_p\leq  c'_pNt_0$
with $N\leq 6$ and $t_0\leq \ell_2^{-p}.$
\end{proof}

\vspace{1mm}

\begin{proposition}\label{>2}
If $\alpha > 2$ then \eqref{seom} is not valid on $K^{\alpha}$.
\end{proposition}
\begin{proof}
By \cite{CA}, the polynomials $(\omega_n)_{n=0}^{\infty}$ form an absolute topological basis in the space  ${\mathcal E}(K^{\alpha})$.
If \eqref{seom} holds then
\begin{equation}\label{EO}
W(f)=\sum_{n=0}^{\infty}\xi_n(f)\tilde{\omega}_n
\end{equation}
is a linear continuous extension operator. However, by \cite{ag97}, ${\mathcal E}(K^{\alpha})$ does not have $EP$ for $\alpha > 2$.
\end{proof}

The following example illustrates the above proposition. We will directly show the absence of \eqref{seom} for $\alpha > 2$.
We restrict ourselves to the same value of $\delta$ as in Proposition \ref{=2}. Let us consider the function $\varphi$ in more detail.
Let $\tau(x):=\exp\left(-\frac{1}{x}\right), Q_0(x):= x-1-x^2.$ Then $\varphi'(x)=\varphi(x)\,\tau(x)\,(x-x^2)^{-2}Q_0(x)$ and by induction
\begin{equation}\label{der}
\varphi^{(k)}(x)=\varphi(x)\,\tau(x)\,(x-x^2)^{-2k}Q_{k-1}(\tau(x))
\end{equation}
for $k\geq 2.$ Here, $Q_k(\tau(x))=Q_0(x)Q_{k-1}(\tau(x)) \tau(x)+ Q_{k-1}(\tau(x))r_k(x)+(1-x)^2\frac{dQ_{k-1}}{d\tau} \tau(x)$ with
$r_k(x):=(1-x)^2+2kx(1-x)(2x-1)$ for $k\in \Bbb N.$ From this it follows that the signs of $\varphi^{(k)}(x)$ and $Q_{k-1}(\tau(x))$ coincide and
$Q_k(\tau(1))=(-1)^{k+1} e^{-k}.$ Assume that $k\in \Bbb N$ is even. Then $\varphi^{(k)}(x)>0$ and $\varphi^{(k-1)}(x)<0$ near point 1.

We will denote by $\eta_k$  the number
$\min\{ \eta: |Q_k(\tau(x))-Q_k(\tau(1))|\leq \frac{1}{2e^k} \,\,\,\,\,\mbox {for}\,\,\,\,\,\,\eta\leq x\leq 1\}.$
Let $\theta_k:=\max\{ \eta_k, \eta_{k-1}, 1-\frac{1}{4\sqrt{ke}}\}.$ Then

\begin{equation}\label{at 1}
\theta_k\,Q_{k-1}(\tau(\theta_k))-2k (\theta_k-\theta_k^2)^2 |Q_{k-2}(\tau(\theta_k))|>\frac{1}{2e^k}.
\end{equation}
Indeed, $\theta_k>\frac{7}{8}$ for $k\geq 2,\,Q_{k-1}(\tau(\theta_k))\geq \frac{1}{2e^{k-1}}$ and $|Q_{k-2}(\tau(\theta_k))|\leq \frac{3}{2e^{k-2}}.$
Also, $\theta_k^2<1$ and $(1-\theta_k)^2\leq \frac{1}{16 ke}.$ Therefore, the left side of \eqref{at 1} exceeds
$\frac{7}{16 e^{k-1}} -\frac{3}{16 e^{k-1}},$ which gives the desired inequality. Let
$A_k:=\varphi(\theta_k)\tau(\theta_k)(\theta_k-\theta_k^2)^{-2k}\frac{1}{2e^k}.$

{\bf Example.} Assume that $\alpha > 2$. Let $N=2^s, \delta=\ell_s.$ Then there is $p$ and $z$ such that for each $q$ and $C$ we have for sufficiently large $s$
\begin{equation}\label{no62}
|\tilde{\omega}^{(p)}_N(z)|>C |\omega_N|_q.
\end{equation}
Let, as above, $Z=(x_k)_{k=1}^N, d_k(x):=d_k(x,Z).$ In view of the structure of the set, we have $d_k(1)\geq d_k(x)$ for each $x\in K^{\alpha}$ if $k\geq 2,$
so, by \eqref{om}, $|\omega_N|_q\leq N^q \prod_{k=q+1}^N d_k(1)$ for $q\geq 1.$

Let us fix an even $p>\frac{2\alpha-3}{\alpha-2}$ and $z=1+\theta_p \delta$. We aim to estimate $|\tilde{\omega}^{(p)}_N(z)|$ from below.
Since $u_{\delta}(x)=\varphi(\frac{x-1}{\delta})$ for $1<x<1+\delta,$ we have $u_{\delta}^{(j)}(z)=\ell_s^{-j} \varphi^{(j)}(\theta_p).$
It follows that
$$ \tilde{\omega}^{(p)}_N(z)=\omega_N(z)\,\ell_s^{-p}\varphi^{(p)}(\theta_p)+ p\omega'_N(z)\,\ell_s^{-p+1}\varphi^{(p-1)}(\theta_p)+
\sum_{j=2}^p  \binom{p}{j} \omega^{(j)}_N(z)\ell_s^{-p+j}\varphi^{(p-j)}(\theta_p).$$
Since $p$ is even, the first term on the right-hand side is positive and the second is negative.
In our case, all values $\omega^{(j)}_N(z)$ are positive. In particular, $\omega_N(z)=\prod_{k=1}^N d_k(z)=\theta_p \ell_s \cdot \prod_{k=2}^N d_k(z)$
and $\omega'_N(z)=\sum_{i=1}^N \prod_{k=1,k\ne i}^N d_k(z)=\prod_{k=2}^N d_k(z) [1+d_1(z)\sum_{i=2}^N d^{-1}_k(z)],$  where the expression in square
brackets is smaller than 2, as is easy to check. By \eqref{om}, $\omega^{(j)}_N(z)\leq N^j \prod_{k=j+1}^N d_k(z),$ so the modulus of the last sum in
the above representation of $ \tilde{\omega}^{(p)}_N(z)$ can be estimated from above by $C_1\sum_{j=2}^p t_j,$ where
$C_1=\max_{0\leq j \leq p-2}|\varphi|_{j,[0,1]}$ and $t_j= \binom{p}{j}N^j\,\ell_s^{-p+j} \prod_{k=j+1}^N d_k(z).$ It is a simple matter to show
$2 t_{j+1}<t_j.$ Hence, $C_1\sum_{j=2}^p t_j\leq 2\,C_1\,t_2.$ Therefore,
$$ \tilde{\omega}^{(p)}_N(z)\geq \ell_s^{-p+1} \prod_{k=2}^N d_k(z)[ \theta_p \varphi^{(p)}(\theta_p) -2p\varphi^{(p-1)}(\theta_p)-2\,C_1p^2N^2\ell_s d_3^{-1}(z)].$$
By \eqref{der}, $ \theta_p \varphi^{(p)}(\theta_p) -2p\varphi^{(p-1)}(\theta_p)= A_p.$ The last term in brackets is arbitrarily small for sufficiently large $s$.
It remains to prove that for each $C$ and $q$ there is $s_0$ such that $\ell_s^{-p+1} \prod_{k=2}^N d_k(z)>C N^q \prod_{k=q+1}^N d_k(1)$ for $s>s_0$.
 Of course, $d_k(z)>d_k(1)$ for each $k$, so we reduce the desired inequality to $C N^q \ell_s^{p-1}<\prod_{k=2}^q d_k(1).$
Let $2^{w-1}<q\leq 2^w$ for some $w$.  Then $\prod_{k=2}^q d_k(1)\geq \prod_{k=2}^{2^w} d_k(1).$ Here, $d_2(1)=\ell_{s-1}$ In its turn,  $\ell_{s-2}-\ell_{s-1}\leq d_3,d_4\leq \ell_{s-2}, \ldots,  \ell_{s-w}-\ell_{s-w+1}\leq d_{2^{w-1}+1},\ldots, d_{2^w} \leq \ell_{s-w}.$
Thus $\prod_{k=2}^{2^w} d_k(1)\geq \pi_w \cdot \ell_{s-1}\ell^2_{s-2} \cdots \ell^{2^{w-1}}_{s-w},$ where
$\pi_w=\prod_{k=1}^{w-1}\left(1-\frac{\ell_{s-k}}{\ell_{s-k-1}}\right)^{2^k}.$ An easy computation shows that $\pi_w>\frac{1}{2},$ so \eqref{no62} holds if
$$ 2C 2^{sq} \ell_s^{p-1}<\ell_{s-1}\ell^2_{s-2} \cdots \ell^{2^w}_{s-w}.$$
Here, the right side is $\ell_1^{\kappa}$ with $\kappa=\alpha^{s-2}+2\,\alpha^{s-3}+\cdots+ 2^{w-1}\,\alpha^{s-w-1}=\alpha^{s-2}[1+ \frac{2}{\alpha}+\cdots
+(\frac{2}{\alpha})^{w-1}]<\frac{\alpha^{s-1}}{\alpha-2}.$ On the other hand, $\ell_s^{p-1}=\ell_1^{(p-1)\alpha^{s-1}}$ with
$(p-1)\alpha^{s-1}-\kappa>\alpha^{s-1},$ due to the choice of $p$. Clearly, $2C 2^{sq} \ell_s<1$ for sufficiently large $s$.

\section{Continuity of Paw{\l}ucki-Ple\'{s}niak's operator for a non-Markov set}

Let points $Z_N=(x_k)_{k=1}^N$ be chosen in $K^2$ by the rule of increase of type and
\begin{equation} \label{Op}
W(f,x)= L_1(f,x;Z_1)\cdot u_{\delta_1}(x)+\sum _{N=1}^{\infty}[L_{N+1}(f,x;Z_{N+1})- L_N(f,x;Z_N)]\cdot u_{\delta_N}(x),
 \end{equation}
where $\delta_N=\ell_s$ for $2^s\leq N<2^{s+1}.$\\

Compare the proofs of the following theorem and T.\ref{WP}.
\begin{theorem} \label{GP}
The operator $W:{\mathcal E}(K^2) \longrightarrow C^{\infty}([-2,2])$ is bounded.
\end{theorem}
\begin{proof} Let $G_N$ denote the $N-$th term of the series \eqref{Op}. Fix $f\in {\mathcal E}(K^2), p\in \Bbb N$ and $x\in \Bbb R.$
By Newton's form of the interpolation operator, $G_N(f,x)=\xi_N(f)\tilde{\omega}_N.$ We aim to show that $|G_N^{(p)}(f,x)|$ is a term
of a series that converges uniformly with respect to $x$. Proposition \ref{=2} gives $q$ and $C$ with
$|G_N^{(p)}(f,x)|\leq C\,|\xi_N(f)|\cdot|\omega^{(q)}_N|_0\leq C \, M_N^{(q)}\cdot |\xi_N(f)|\cdot|\omega_N|_0.$
Here and below, to simplify the writing, we omit the argument $K^2$ of $ M_N^{(q)}(\cdot), \Lambda_j(\cdot),$ and $E_j(f,\cdot).$

Next, $|\xi_N(f)|\cdot|\omega_N|_0=|L_{N+1}(f)- L_N(f)|\leq (\Lambda_{N+1}+1)E_N(f)+ (\Lambda_N+1)E_{N-1}(f)$ by the argument from Lebesgue's Lemma, so
$|\xi_N(f)|\cdot|\omega_N|_0\leq (\Lambda_{N+1}+\Lambda_N+2)E_{N-1}(f).$
Applying Proposition \ref{Leb} yields $\Lambda_{N+1}+\Lambda_N+2\leq 2h_0^{-N} N,$ as is easy to check. Therefore,
$$ |G_N^{(p)}(f,x)|\leq 2C\,N\,h_0^{-N}\, M_N^{(q)}\cdot E_{N-1}(f).$$

By Proposition \ref{mf-jt}, for each $r>q$ there exists $r_1$ such that
$$M_N^{(q)}\cdot E_{N-1}(f)\leq \rho_{q+1}(N+1)\cdots \rho_r(N+1)\,||\,f\,||_{\,r_1}.$$
Arguing in the same way as in the proof of Proposition  \ref{mf-jt}, we can take a sufficiently large number of terms $\rho_j(N+1)$ in such a way that
$\rho_{q+1}(N+1)\cdots \rho_m(N+1)\leq \ell_s$ and $2C\,N\,h_0^{-N}\,\rho_{m+1}(N+1)\cdots \rho_r(N+1)\leq 1.$ This gives
$ |G_N^{(p)}(f,x)|\leq \ell_s\,||\,f\,||_{\,r_1}.$ Of course, the series $\sum_{s=1}^{\infty} 2^s \ell_s$ converges.
 \end{proof}

{\bf Conclusions.}
At least for the considered case,\\
1. The operator $W$ is continuous not only in $\tau_J$, but also in the stronger Whitney topology.\\
2. This can be shown by a modification of Ple\'{s}niak's argument.\\
3. The difference between sets $K^{\alpha}$ with and without the extension property does not depend on the growth rate of Markov's factors but
rather on the existence of suitable individual extensions of $\omega_N$.\\
4. Since these polynomials form a topological basis in the corresponding Whitney space, the operator \eqref{Op} given by the interpolation method 
coincides with the operator \eqref{EO} obtained by extensions of the basis elements. This method goes back to Mityagin \cite{mi}.\\

This coincidence can be clearly observed in the following model case.

{\bf Example.} Fekete points $\mathcal X_n$  are known for $K=[-1,1]$, see, e.g., \cite{Sz}, p.382. They are zeros of $\omega_n(x):=(1-x^2) P_{n-2}^{1,1}(x),$ where $P_{n-2}^{1,1}$ is the Jacobi polynomial with the parameters $\alpha=\beta=1.$ By \cite{B}, for any admissible parameters, the Jacobi polynomials
form a basis in $C^{\infty}[-1,1]$. The value $\delta_n=n^{-2}$ defines $\tilde{\omega}_n$ and the extension operator \eqref{Op} that can be written as
$W(f)=\sum_{n=1}^{\infty}\xi_N(f)\tilde{\omega}_N.$

\end{document}